\documentclass[11pt]{amsart}
\usepackage{verbatim,amssymb,amsmath,graphicx,hyperref,nicefrac,bbm}
\usepackage[usenames]{color}

  \scrollmode

\newcommand{\R}{{\mathbb R}}
\newcommand{\N}{{\mathbb N}}

\newcommand{\EE}{{\mathbb E}}
\newcommand{\PP}{{\mathbb P}}
\newcommand{\1}{{\mathbbm 1}}

\theoremstyle{plain}
\newtheorem{theorem}{Theorem}

\newtheorem{lemma}{Lemma}

\theoremstyle{definition}

\newtheorem{ex}{Example}

\begin{document}
\title[Approximation of SDEs with smooth coeffients with derivatives of linear growth ]{A note on strong approximation of SDE{\footnotesize s} with smooth coefficients that have at most linearly growing derivatives}

\author[M\"uller-Gronbach and Yaroslavtseva]
{Thomas M\"uller-Gronbach and Larisa Yaroslavtseva}
\address{
Fakult\"at f\"ur Informatik und Mathematik\\
Universit\"at Passau\\
Innstrasse 33 \\
94032 Passau\\
Germany} \email{larisa.yaroslavtseva@uni-passau.de, thomas.mueller-gronbach@uni-passau.de}

\begin{abstract}
Recently, it has been shown in [Jentzen, A., M\"uller-Gronbach, T., and Yaroslavtseva, L., Commun. Math. Sci., 14, 2016] that there exists a system of autonomous
 stochastic differential equations (SDE) on the time interval $[0,T]$ with infinitely  differentiable and bounded coefficients  such that no strong approximation method 
based on evaluation of the
driving Brownian motion at finitely many fixed times 
   in $[0,T]$, e.g. on an equidistant grid, 
    can converge in absolute mean to the solution at the final time with a polynomial rate in terms of the number of Brownian motion values that are used.
In the literature on strong approximation of SDEs, polynomial error rate results are typically
 achieved under the assumption that the first order derivatives of the coefficients of the equation satisfy a polynomial growth condition. This assumption is violated for the 
    pathological  SDEs from  the above mentioned negative result.
    However, in the present
 article we  construct an SDE with smooth coefficients that have first order derivatives of at most linear growth such  that  the solution at the final time can not be approximated with a polynomial rate, whatever method based on observations of the
driving Brownian motion at finitely many fixed times is used. 
Most interestingly, it turns out that using a method that adjusts the number of evaluations of the driving Brownian motion to its actual path, the latter SDE can be approximated  with rate 1 in terms of the average number of evaluations that are used. To the best of our knowledge, this is only the second example in the literature of an SDE for which there exist adaptive methods that perform superior 
to non-adaptive ones with respect to the convergence rate. 
\end{abstract}

\maketitle

\section{Introduction}

Let $d,m\in \N$, $T\in(0, \infty)$, consider a $d$-dimensional system of autonomous stochastic differential equations (SDE)
\begin{equation}\label{sde0}
\begin{aligned}
dX(t) & = \mu(X(t)) \, dt + \sigma(X(t)) \, dW(t), \quad t\in [0,T],\\
X(0) & = x_0
\end{aligned}
\end{equation}
with a deterministic initial value $x_0\in\R^d$, a drift coefficient $\mu\colon\R^d\to\R^d$, a diffusion coefficient $\sigma\colon \R^d\to\R^{d\times m}$ and an $m$-dimensional driving Brownian motion $W$, and assume that \eqref{sde0} has a unique strong solution $(X(t))_{t\in[0, T]}$. 
A fundamental problem in the numerical analysis of SDEs is to characterize when 
the solution at the final time $X(T)$ can be approximated  with a polynomial 
error rate based on  finitely many evaluations of the driving Brownian motion $W$ 
in terms of explicit regularity conditions on the coefficients $\mu$ and $\sigma$.

It is well-known that if the coefficients $\mu$ and $\sigma$ are globally Lipschitz continuous then the classical Euler-Maruyama scheme achieves the rate of convergence $1/2$, see \cite{m55}. Moreover,
the recent literature on numerical approximation of SDEs contains a number
of results on approximation schemes that are specifically designed for 
     SDEs with
non-Lipschitz coefficients and achieve
polynomial convergence rates 
under weaker conditions on $\mu$ and $\sigma$, 
see e.g.
\cite{h96,hms02,%Schurz2006,
HutzenthalerJentzenKloeden2012, MaoSzpruch2013Rate,
WangGan2013,Sabanis2013ECP,TretyakovZhang2013,Beynetal2014,KumarSabanis2016, Beynetal2016, FangGiles2016}
for SDEs with globally monotone coefficients
and e.g.
\cite{BerkaouiBossyDiop2008,GyoengyRasonyi2011,DereichNeuenkirchSzpruch2012,
Alfonsi2013,NeuenkirchSzpruch2014,HutzenthalerJentzen2014,
HutzenthalerJentzenNoll2014CIR, 
LS15b,LS16, Tag16, HH16b} %ChassagneuxJacquierMihaylov2014,
for SDEs with possibly non-monotone coefficients.

On the other hand, it has recently been shown in \cite{JMGY15}  that for any sequence  $(a_n)_{n\in\N}\subset (0,\infty)$, which may converge to zero arbitrarily slowly,  there exists an SDE \eqref{sde0} with $d=4$ and $m=1$ and with infinitely  differentiable and bounded coefficients $\mu$ and $\sigma$ such that no sequence of approximations $\widehat X_n(T)$ of $X(T)$, where $\widehat X_n(T)$ is
based on $n$ evaluations of the driving Brownian
motion $W$ at fixed time points in $[0,T]$, can converge to $X(T)$ in absolute mean  faster than the given sequence $(a_n)_{n\in\N}$. 
More formally, for this SDE one has for every $n\in\N$,
\begin{equation}
\label{eq:intro3} 
  \hspace{-0.5cm} \inf_{
    \substack{
      s_1,\dots,s_n\in [0,T]
    \\
     u\colon \R^n\to\R^4 \text{ measurable}
    }
  }\hspace{-0.5cm}
  \EE\bigl[
    |
      X( T )
      -
      u( W( s_1 ), \dots, W( s_n )
      )
    |\bigr]
\geq  a_n.
\end{equation}
In~\cite{GJS17} it has been proven 
that the negative result~\eqref{eq:intro3} can even be achieved  with $m=1$ and $d=2$ in place of $d=4$. 
In particular, \eqref{eq:intro3} implies that there exists an SDE \eqref{sde0} with infinitely  differentiable and bounded coefficients $\mu$ and $\sigma$ such that
its solution at the final time can not be approximated with a polynomial mean error rate  based on 
evaluations of the driving Brownian
motion $W$ at finitely many fixed time points in $[0,T]$, i.e., for every $\alpha>0$,
\begin{equation}\label{lnegres}
\lim_{n\to\infty}\Bigl( n^\alpha\cdot \hspace{-0.5cm} \inf_{
    \substack{
      s_1,\dots,s_n\in [0,T]
    \\
     u\colon \R^{nm}\to\R^d \text{ measurable}
    }
  }\hspace{-0.5cm}
\EE \bigl[|X(T)-u(W(s_1), \ldots,
W(s_n))|\bigl]\Bigr) =\infty.
\end{equation}
 We add that the latter statement for the special case when the approximation $u(W(s_1), \ldots,
W(s_n))$ is given by the Euler-Maruyama scheme with 
time step $1/n$
has first been shown in \cite{hhj12}.  

The proof of the negative result \eqref{eq:intro3} in \cite{JMGY15} is constructive. Each of the respective SDEs  is given by $X(0)=0$ and  
\begin{equation}\label{hardsde}
\begin{aligned}
dX_1(t) & = dt,\quad dX_2(t)  = f(X_1(t))\,dW(t),\quad
dX_3(t)  = g(X_1(t))\,dW(t),\\
dX_4(t) & = h(X_1(t)) \cdot \cos\bigl(X_2(t)\cdot \psi(X_3(t))\bigr)\,dt
\end{aligned}
\end{equation}
for $t\in [0,T]$, where $f,g,h\colon\R\to\R$ are infinitely  differentiable, bounded, nonzero and satisfy $\{f \neq 0\}\subset (-\infty, \tau_1]$, 
$\{g \neq 0\}\subset [\tau_1, \tau_2]$, $\{h\neq 0\}\subset [\tau_2, T]$, $\int_{\tau_2}^T h(t)\, dt \neq 0$ and $\inf_{x\in [0,\tau_1/2]} |f'(x)| >0$ for some $0<\tau_1<\tau_2<T$, and $\psi\colon\R\to (0,\infty)$ is infinitely differentiable, strictly increasing and satisfies $\lim_{x\to \infty} \psi(x) = \infty$. Under these assumptions 
the fourth component of the solution of the SDE \eqref{hardsde} at the final time is given by 
\begin{equation}\label{fourthcomp}
X_4(T)= \cos\Bigl(\int_{0}^{\tau_1}f(t)\,dW(t)\cdot\psi\Bigl(\int_{\tau_1}^{\tau_2}g(t)\,dW(t)\Bigr)\Bigr)\cdot \int_{\tau_2}^T h(t)\,  dt
\end{equation}
and there exist $c_1,c_2,c_3\in (0,\infty)$ such that for every $n\in\N$,
\begin{equation}\label{lb1}
  \inf_{
    \substack{
      s_1,\dots,s_n\in [0,T]
    \\
     u\colon \R^n\to\R \text{ measurable}
    }
  \hspace{-0.4cm}
  }
  \EE\bigl[
    |
      X_4( T )
      -
      u\big( W( s_1 ), \dots, W( s_n )
      \big)
    |\bigr]
\geq  c_1\cdot\exp\bigl(-c_2\cdot (\psi^{-1}(c_3 \cdot n^{3/2})^2\bigr),
\end{equation}
see Corollary 4.1 in~\cite{JMGY15}.

   It follows from~\eqref{lb1} that
if 
\begin{equation}\label{exp}
\forall q\in (0,\infty)\colon \,\,\lim_{x\to\infty} \exp(-q x^2)\cdot \psi (x) = \infty
\end{equation}
then a polynomial  rate of convergence to zero of the left hand side in~\eqref{lb1} can not be achieved, see Corollary 4.2 in~\cite{JMGY15}.
On the other hand it is straightforward to check that the equidistant Euler-Maruyama scheme for the SDE~\eqref{hardsde}
 achieves a polynomial mean error rate if the derivative  $\psi'$ of $\psi$ is of at most polynomial growth.
The latter two facts are reflected in the growth properties of the first order
 derivatives of the coefficients $\mu$ and $\sigma$ of the SDE~\eqref{hardsde}. 
 All of
the first order 
  derivatives of $\mu$ and $\sigma$ are globally bounded, up to the  derivatives 
\[
\frac{\partial \mu_4}{\partial x_2}(x) =  -h(x_1)\cdot\psi(x_3)\cdot\sin(x_2\cdot\psi(x_3)),\quad \frac{\partial \mu_4}{\partial x_3}(x) =  -h(x_1)\cdot x_2\cdot \psi'(x_3)\cdot \sin(x_2\cdot\psi(x_3)),
\] 
which are both of at most polynomial growth if and only if $\psi'$  is of at most  polynomial growth.

For the vast majority of SDEs with locally Lipschitz continuous coefficients used for modelling in applications it holds that the first order  derivatives of the coefficients are of at most polynomial growth.
Moreover, a polynomial growth condition on the 
first order 
derivatives of the coefficients of an SDE is 
one of the standing assumptions
in the literature
 when polynomial mean error rates are obtained under 
monotonicity 
 conditions, see e.g.
 \cite{h96,hms02,%Schurz2006,
HutzenthalerJentzenKloeden2012, 
WangGan2013,Sabanis2013ECP,TretyakovZhang2013,Beynetal2014,KumarSabanis2016, Beynetal2016, FangGiles2016}. 
         Therefore it is 
important to investigate whether a sub-polynomial rate of convergence as in \eqref{lnegres} may also happen  when the first order 
 derivatives of the coefficients are of at most polynomial growth.

This question can easily be answered with a yes. For the choice $\psi(x) = \exp(x^3)$, which satisfies~\eqref{exp}, the random variable $X_4(T)$ in~\eqref{fourthcomp} can also be obtained as the fifth component of the solution at the final time of an SDE given by $Y(0)=(0,0,0,1,0)$
 and
\begin{equation}\label{hardsde2}
\begin{aligned}
dY_1(t)&=dt,\quad dY_2(t) =f(Y_1(t))\, dW(t),\quad dY_3(t) =g(Y_1(t))\, dW(t),\\
dY_4(t)&=u(Y_1(t)) \cdot Y_3^3 (t)\cdot Y_4(t)\,dt, \quad 
dY_5(t)=v(Y_1(t)) \cdot \cos\bigl(Y_2(t)\cdot Y_4(t)\bigr)\, dt  
\end{aligned}
\end{equation}
for $t\in [0,T]$,
where $f$, $g$, $h$ satisfy the conditions stated below the SDE~\eqref{hardsde} and, additionally, $f'$ is bounded,  
and $u,v\colon \R\to\R$ are infinitely differentiable and satisfy  
$\{u\neq 0\}\subset [\tau_2, \tau_3]$,  $\{v\neq 0\}\subset
[\tau_3, T]$, $\int_{\tau_2}^{\tau_3} u(s)\, ds=1$ and 
$\int_{\tau_3}^T v(s)\, ds=\int_{\tau_2}^T h(t) dt$ for some $\tau_3\in (\tau_2,T)$. Clearly, the coefficients of the SDE~\eqref{hardsde2} have first order derivatives of at most polynomial growth and $Y_5(T) = X_4(T)$.

Note, however, that in contrast to the solution $X$ of the SDE~\eqref{hardsde}, the solution $Y$ of the SDE~\eqref{hardsde2} is not integrable at any time 
$t\in[\tau_3, T]$.
In fact, it is easy to see that
\[
\EE[\sup_{t\in[0, T]}|X(t)|] < \infty,\quad \inf_{t\in[\tau_3, T]}\EE[|Y_4(t)|] = \infty.
\]
It therefore seems reasonable to modify the question posed above and 
to
ask whether 
a sub-polynomial rate of convergence as in \eqref{lnegres} may also happen
for an SDE~\eqref{sde0}
 that has smooth coefficients with first order  derivatives of at most polynomial growth and a  solution $X$ with  
\begin{equation}\label{l102}
\EE[\sup_{t\in[0, T]}|X(t)|] < \infty.
\end{equation} 
 
In the actual paper we show that the answer to this question is positive as well. More precisely,
consider the $7$-dimensional SDE given by $X(0)=(0,0,0,0,1,0,0)$ and 
\begin{equation}\label{hardsde3}
\begin{aligned}
dX_1(t)&=dt,\quad dX_2(t) =f(X_1(t))\, dW(t),\\
 dX_3(t) &  =f^2(X_1(t))\,dt + 2 X_2(t)\cdot f(X_1(t))\, dW(t),\\
dX_4(t)&= \tfrac{1}{4} g'(X_1(t))\cdot X_3(t)\, dt,\quad 
dX_5(t) =X_4(t)\cdot X_5(t)\, dt,\\
dX_6(t) & = \tfrac{h'(X_1(t))\cdot X_5(t)}{(1+X_2^2(t))^{\frac{1}{2}}\cdot \ln^2(2 + X_2^2(t) )},\quad dX_7(t) = X_5(t)\cdot X_6(t)\, dt 
\end{aligned}
\end{equation}
for $t\in [0,T]$, where $f,g,h\colon\R\to\R$ satisfy the conditions stated below the SDE~\eqref{hardsde} and, additionally,
 $g,h\ge 0$,
$f'$ is bounded
  and   $\int_0^{\tau_1} f^2(t)\, dt = \int_{\tau_1}^{\tau_2} g(t)\, dt = \int_{\tau_2}^T h(t)\, dt =1$.
  See Example~\ref{ex:fgh} for a possible choice of $f,g,h$. 
 The assumptions on the functions $f,g$ and $h$ imply that 
 all of the first order  derivatives of 
 the coefficients of the SDE~\eqref{hardsde3} are of at most linear growth and the solution $X$ of the SDE~\eqref{hardsde3} satisfies the moment condition \eqref{l102}, see Lemmas~\ref{Tlemma1} and \ref{Tlemma2}. Moreover, as a consequence of
Theorem~\ref{t1} we obtain
that there exists $c\in (0,\infty)$ such that for all $n\in\N$,
\begin{equation}\label{newlb}
 \inf_{
    \substack{
      s_1,\dots,s_n\in [0,T]
    \\
     u\colon \R^n\to\R^7 \text{ measurable}
    }
  \hspace{-0.4cm}
  }
  \EE
    \bigl[|
      X( T )
      -
      u\big( W( s_1 ), \dots, W( s_n )
      \big)
    |\bigr]
\geq  c\cdot \frac{1}{\ln^2(n+1)},
\end{equation}
and therefore $X(T)$ can not be approximated with a polynomial mean error rate  based on 
evaluations of the driving Brownian
motion $W$ at finitely many fixed time points in $[0,T]$. To the best of our knowledge this is the first result in the literature, which shows that a sub-polynomial rate of convergence may happen even then when the first order derivatives of the coefficients are of at most polynomial growth.   It  implies in particular that for such SDEs 
      even 
 tamed or projected versions of the Euler-Maruyama scheme or the Milstein scheme, which are specifically designed to         cope with the case of 
superlinearly growing coefficients, see e.g. \cite{
HutzenthalerJentzenKloeden2012, Sabanis2013ECP,Beynetal2014,KumarSabanis2016, Beynetal2016}
 may fail to achieve a polynomial convergence rate.

The negative result~\eqref{newlb}
covers only approximations that are based on $n$ evaluations of the driving Brownian motion $W$ at fixed time points $s_1,\dots,s_n\in[0,T]$ and leaves it open whether a polynomial mean error rate can  be achieved by employing approximations that may 
    adapt 
the number as well as the location of the evaluation sites of
$W$ 
to the actual path of $W$, e.g. by  numerical schemes that adjust the actual step size according to a
criterion that is based on the values of  $W$ observed so far, see e.g.
\cite{Gaines1997,MG02_habil,m04, Moon2005, RW2006, LambaMattinglyStuart2007, Hoel2012,Hoel2014} and the references therein for methods of this type. However, it is
well-known that for a huge class of SDEs \eqref{sde0} with  globally Lipschitz continuous coefficients $\mu$ and $\sigma$  adaptive approximations of the latter type  can not achieve a better  rate of convergence compared to what is
best possible for non-adaptive ones, which at the same time coincides with the best possible rate of convergence that can be achieved by approximations based on evaluating  $W$ at $n$ equidistant times,
see \cite{MG02_habil,m04} and the discussion on asymptotic constants therein. Moreover, it has recently been shown in~\cite{Y17} that the SDE~\eqref{hardsde} with $\psi$ satisfying~\eqref{exp} 
can not be approximated with
a polynomial mean error rate even then when adaptive approximations may be used.

Up to now there seems to be only 
one
example of an SDE known in the literature, 
   for which
adaptive approximations are superior to non-adaptive ones with respect to the convergence rate. 
In~\cite{HH16} it has been shown that for the one-dimensional squared
Bessel process, i.e. the solution of the SDE~\eqref{sde0} with $d = m =1$, $\mu=1$ and $\sigma(x)=2\sqrt{|x|}$, any non-adaptive approximation
 of $X(T)$ based on $n$ equidistant evaluations of $W$ can
  only achieve a mean error rate of order $1/2$ in terms of $n$, while for every $\alpha \in (0,\infty)$ there exist $c\in (0,\infty)$ and a sequence of approximations $\widehat X_n(T)$, each based on $n$ sequentially chosen evaluations of $W$, such that $\EE[|X(T)- \widehat X_n(T)|] \le c\cdot n^{-\alpha}$. 

Interestingly it turns out that the SDE \eqref{hardsde3} provides the second 
example    after~\cite{HH16}  
of an SDE in the literature, for which  there exist adaptive approximations 
      that perform superior to non-adaptive ones with respect to the convergence rate.     
Indeed,  there exists $c\in (0,\infty)$ and  a sequence of approximations $\widehat X_n(T)$, each based on $n$ sequentially chosen evaluations of $W$ on average, such that for all $n\in\N$,
\[
\EE[|X(T)- \widehat X_n(T)|] \le c\cdot n^{-1},
\]
see Theorem \eqref{lt2}. 

We briefly describe the content of the paper. 
In Section~\ref{sec:setting} we introduce the particular SDE with smooth coefficients that
is studied in the present
paper and we discuss moment properties of its solution.  
Our main results on a sub-polynomial lower error bound for non-adaptive methods  (Theorem~\ref{t1}) and a polynomial upper error bound for a suitable adaptive method (Theorem~\ref{lt2}) are stated in Section~\ref{results}. The respective proofs are   carried out in Sections~\ref{proof1} and~\ref{proof2}. Section~\ref{discuss} is devoted to a discussion of our results and naturally arising open questions.

\section{An SDE with smooth coefficients that have at most linearly growing derivatives}
\label{sec:setting}

Throughout this article we fix the following setting.

Let 
$ T  \in (0,\infty) $,
let
$ ( \Omega, \mathcal{F}, \PP ) $
be a probability space with a normal filtration
$ ( \mathcal{F}_t )_{ t \in [0,T] } $
and let
$
  W \colon [0,T] \times \Omega \to \R
$
be a
standard $ ( \mathcal{F}_t )_{ t \in [0,T] } $-Brownian motion
on $ ( \Omega, \mathcal{F}, \PP ) $.

Let  $ 0<\tau_1<\tau_2<T$ and let $ f, h, g \in C^{ \infty }(\R, \R) $ satisfy 
\begin{equation}\label{assum1}
 \{f\neq 0\} \subseteq ( - \infty, \tau_1 ],\quad
 \{g\neq 0\} \subseteq [ \tau_1, \tau_2 ],\quad
 \{h\neq 0\}  \subseteq [ \tau_2, T]
\end{equation}
 as well as 
\begin{equation}\label{assum2} 
  \sup_{t\in ( - \infty, \tau_1 ]}| f(t)| <\infty,\quad \sup_{t\in ( - \infty, \tau_1 ]}| f'(t)| <\infty,\quad \inf_{ t\in [ 0, \tau_1/2  ] } | f'(t) | > 0,\quad g\geq 0,\quad h\ge 0
\end{equation}
  and
\begin{equation}\label{assum3}
\int_{0}^{\tau_1} f^2(t)\, dt=
\int_{\tau_1}^{\tau_2} g(t)\, dt=
 \int_{ \tau_2 }^{ T } h(t) \, dt =1.
\end{equation}
See the following example for a possible choice of $f,g$ and $h$.
\begin{ex}
\label{ex:fgh} Define $ \widetilde f, \widetilde g, \widetilde h \colon
\R \to \R $ by
\begin{equation}\label{bsp}
\begin{aligned}
  \widetilde f(x)  & = \1_{(-\infty,\tau_1)}(x)\cdot
    \exp\bigl(      
      \tfrac{
        1
      }{
        x - \tau_1
      }
    \bigr),\\
  \widetilde g(x) &  = \1_{(\tau_1,\tau_2)}(x)\cdot
    \exp\bigl(      
      \tfrac{ 1 }{ \tau_1 - x } + \tfrac{ 1 }{ x - \tau_2 }
    \bigr),\\
\widetilde h(x) &  =
    \1_{(\tau_2,T)}(x)\cdot
    \exp\bigl(      
      \tfrac{ 1 }{ \tau_2 - x } + \tfrac{ 1 }{ x - T }
    \bigr). 
    \end{aligned} 
    \end{equation}
Then the functions 
\begin{align*}
f&=\Bigl(\int_{0}^{\tau_1}\widetilde (f(s))^2\,ds\Bigr)^{-1/2}\cdot \widetilde f , \quad 
g= \Bigl(\int_{\tau_1}^{\tau_2}\widetilde g(s)\,ds\Bigr)^{-1}\cdot \widetilde g,\quad h= \Bigl(\int_{\tau_1}^{\tau_2}\widetilde h(s)\,ds\Bigr)^{-1}\cdot \widetilde h 
\end{align*} 
satisfy $f,g,h\in C^{ \infty }(\R, \R) $ as well as the conditions~\eqref{assum1}-\eqref{assum3}.
\end{ex}

Let $p\in [1, \infty)$ and define
$
  \mu, \sigma \colon \R^7 \to \R^7
$
as well as $x_0\in \R^7$ by
\begin{equation}\label{coeff}
\begin{aligned}
  \mu(x) & = \Bigl( 1, 0, f^2(x_1), \tfrac{g'( x_1 )}{4p}  \cdot x_3, x_4\cdot x_5, \tfrac{ h'( x_1 ) \cdot x_5}{(1+x_2^2)^{\frac{1}{2p}}\cdot \ln^{\frac{2}{p}} (2+x_2^2)},x_5\cdot x_6 \Bigr),\\
  \sigma(x) &=   \bigl(    0, f( x_1 ) , 2x_2\cdot f(x_1), 0, 0,0,0  \bigr),\\
  x_0&=(0,0,0,0,1,0,0).
\end{aligned}
\end{equation}

\begin{lemma}\label{Tlemma1} We have $\mu,\sigma\in C^\infty(\R^7,\R^7)$. Moreover, there exists $c\in (0,\infty)$ such that for all $x\in\R^7$,
\[
\sum_{i,j=1}^7 \bigl( \bigl |\tfrac{\partial \mu_i}{\partial x_j}(x)\bigr| + \bigl |\tfrac{\partial \sigma_i}{\partial x_j}(x)\bigr|\bigr)\leq c\cdot (1+|x|).
\]
\end{lemma}
\begin{proof}
Infinite differentiability of $\mu$ and $\sigma$ is an immediate consequence of the definition of these functions and the fact that $f,g,h\in C^\infty(\R,\R)$. Moreover, it is straightforward to check that there exists $c\in (0,\infty)$ such that for all $i,j\in\{1,\dots,7\}$  and $x\in\R^7$,
\[
\bigl|\tfrac{\partial \sigma_i}{\partial x_j}(x)\bigr| \le c\cdot(|f(x_1)|+|f'(x_1)|)\cdot (1+|x|)
\]
and
\[
\bigl|\tfrac{\partial \mu_i}{\partial x_j}(x)\bigr| \le c\cdot(|f(x_1)|\cdot |f'(x_1)| + |g'(x_1)| + |g''(x_1)| + |h'(x_1)| + |h''(x_1)| + 1)\cdot (1+|x|),
\]
which jointly with the fact that $g,h\in C^\infty(\R,\R)$ and 
   the properties~\eqref{assum1} and~\eqref{assum2} 
yields at most linear growth for all first order derivatives of $\mu$ and $\sigma$.
\end{proof}

We study the SDE~\eqref{sde0} with $m=1$, $d=7$ and $x_0,\mu,\sigma$ given by~\eqref{coeff}, i.e. $X(0)=(0,0,0,0,1,0,0)^\top$ and 
\begin{equation}\label{sde}
\begin{aligned}
dX_1(t)&=dt,\quad dX_2(t) =f(X_1(t))\, dW(t),\\
 dX_3(t) &  =f^2(X_1(t))\,dt + 2 X_2(t)\cdot f(X_1(t))\, dW(t),\\
dX_4(t)&= \tfrac{1}{4p} g'(X_1(t))\cdot X_3(t)\, dt,\quad 
dX_5(t) =X_4(t)\cdot X_5(t)\, dt,\\
dX_6(t) & = \tfrac{h'(X_1(t))\cdot X_5(t)}{(1+X_2^2(t))^{\frac{1}{2p}}\cdot \ln^\frac{2}{p}(2 + X_2^2(t) )},\quad dX_7(t) = X_5(t)\cdot X_6(t)\, dt 
\end{aligned}
\end{equation}

Observing~\eqref{assum1} and using It\^{o}'s formula for the component $X_3$ it is straightforward to see that  the equation \eqref{sde} has a unique strong solution given by
\begin{equation}\label{solution}
\begin{aligned}
  X_1(t) & = t ,
\quad
  X_2(t) = \int_0^{ \min( t, \tau_1 ) } f(s) \, dW(s),\quad X_3(t) = X_2^2(t),
\\
 X_4(t) & = \tfrac{1}{4p} X_2^2(\tau_1)\cdot g(t),\quad  X_5(t)  =  \exp\Bigl(\tfrac{1}{4p}\,
  X_2^2(\tau_1)\cdot\int_{0}^{ \min( t , \tau_2 ) } g(s) \, ds \Bigr) ,
\\
  X_6(t) & =
  \tfrac{ X_5(\tau_2)}{(1+X_2^2(\tau_1))^{\frac{1}{2p}}\cdot \ln^{\frac{2}{p}} (2+X_2^2(\tau_1))}
  \cdot
 h(t),\\
 X_7(t) & =     \tfrac{ X_5^2(\tau_2)}{(1+X_2^2(\tau_1))^{\frac{1}{2p}}\cdot \ln^{\frac{2}{p}} (2+X_2^2(\tau_1))}
  \cdot
\int_{0}^{t} h(s)\, ds
\end{aligned}
\end{equation}
for all $ t \in [0,T] $. In particular, by~\eqref{assum3},
\begin{equation}\label{solution2}
\begin{aligned}
  X_1(T) & = T ,\quad  X_2(T) =  X_2(\tau_1) =\int_0^{  \tau_1  } f(s) \, dW(s),
  \quad X_3(T) = X_2^2(T),\\
 X_4(T)& = 0,\quad  X_5(T)  =  \exp\Bigl(\tfrac{1}{4p}\,
  X_2^2(\tau_1)\Bigr), \quad X_6(T) = 0,\\
  X_7(T)  & = \tfrac{ \exp\bigl(\frac{1}{2p}
  X_2^2(\tau_1)\bigr)}{(1+X_2^2(\tau_1))^{\frac{1}{2p}}\cdot \ln^{\frac{2}{p}} (2+X_2^2(\tau_1))}.
\end{aligned}
\end{equation}

Next we discuss integrability properties of the solution $X$.
\begin{lemma}\label{Tlemma2}
We have $X_2(\tau_1)\sim \mathcal{N}( 0,  1)$. Moreover, for all $q\in(0, \infty)$,
\[
\EE \big[\sup_{t\in[0,T]} |X(t)|^q\bigr]<\infty\quad \Leftrightarrow \quad q\le p.
\]
\end{lemma}

\begin{proof}
The first statement follows immediately from the definition of $X_2(\tau_1)$ and the fact that $\EE[X_2^  2(\tau_1)] = \int_{0}^{\tau_1} f^2(t)\, dt= 1$, due to~\eqref{assum3}. Moreover, applying the Burkholder-Davis-Gundy  inequality we obtain that for all $q\in(0, \infty)$ there exists $c\in(0, \infty)$ such that
\begin{equation}\label{mom1}
\EE \big[\sup_{t\in[0,T]} |X_2(t)|^q\bigr] \le c\cdot \Bigl(\int_0^T f^2(t)\, dt\Bigr)^{\tfrac{q}{2}} =c.
\end{equation}
%with a constant $c_q\in (0,\infty)$, which only depends on $q$.
 Employing~\eqref{solution},~\eqref{mom1} and the properties of $g$ we conclude that for all $q\in(0, \infty)$, 
\[
\EE \big[\sup_{t\in[0,T]} |X_4(t)|^q\bigr] = \tfrac{1}{(4p)^q}\cdot\EE[|X_2(\tau_1)|^{2q}]\cdot \sup_{t\in [\tau_1,\tau_2]} (g(t))^q < \infty.
\]
Furthermore,~\eqref{assum2},~\eqref{assum3},~\eqref{solution} and the fact that  $X_2(\tau_1)\sim \mathcal{N}( 0,  1)$ imply that for all $q \in (0,2p)$, 
\[
\EE\big[\sup_{t\in[0,T]} |X_5(t)|^{q}\bigr] = \EE\bigl[\exp(\tfrac{q}{4p}X_2^2(\tau_1))\bigr] = \sqrt{\tfrac{2p}{2p-q}}.
\]
By~\eqref{solution}, the latter  equality and the properties of $h$ we get that for all $q \in (0,2p)$, 
\[
\EE\big[\sup_{t\in[0,T]} |X_6(t)|^{q}\bigr] \le  \tfrac{1}{\ln^{\frac{2q}{p}}(2)}\cdot \EE[|X_5(\tau_2)|^{q}]\cdot \sup_{t\in[\tau_2,T]} h^q(t) <\infty.
\]
Finally, by~\eqref{solution} we see that for all $q\in(0, \infty)$, 
\begin{align*}
\EE\big[\sup_{t\in[0,T]} |X_7(t)|^{q}\bigr] & = \EE\Bigl[\tfrac{ \exp\bigl(\frac{q}{2p}
  X_2^2(\tau_1)\bigr)}{(1+X_2^2(\tau_1))^{\frac{q}{2p}}\cdot \ln^{\frac{2q}{p}} (2+X_2^2(\tau_1))}\Bigr]  = \sqrt{\tfrac{2}{\pi}}\int_0^\infty \tfrac{\exp\bigl(\tfrac{q-p}{2p}\cdot x^2\bigr) }{(1+x^2)^{\frac{q}{2p}}\cdot \ln^{\frac{2q}{p}}(2+x^2)} \, dx,  
\end{align*}
and the latter quantity is finite if and only if $q\le p$.
\end{proof}

\section{Lower and upper error bounds}
\label{results} 

We study strong approximation of the solution $X$ of the equation~\eqref{sde} at the final time $T$.
The following result shows that $X_7(T)$ and thus $X(T)$ as well 
can not be approximated in $p$-th mean sense
 with a polynomial error rate in terms of the number of evaluations of the driving Brownian motion $W$ as long as the number and the location of the evaluation nodes for $W$ are not chosen in a  path-dependent way.

\begin{theorem}
\label{t1}  There exists $c\in(0, \infty)$ such that for all $n\in\N$,
\[
\inf_{
    \substack{
      s_1,\dots,s_n\in [0,T]
    \\
     u\colon \R^n\to\R \text{ measurable}
    }
    }\Bigl(\EE
    \bigl[|
      X_7( T ) -u(W(s_1), \ldots, W(s_n))
    |^p\bigr]\Bigr)^{\tfrac{1}{p}}\geq c\cdot  \frac{1}{\ln^{\frac{2}{p}}(n+1)}.
\]

\end{theorem}

Our next result shows that a polynomial $p$-th mean error rate 
for approximation of $X(T)$
 can be achieved if the number 
 of the evaluation nodes for $W$ is adjusted to the current path of $W$.

For $n\in\N$ we use  $ \overline{W}_n \colon [ 0, \tau_1 ] \times \Omega \to \R $ to denote the piecewise linear interpolation of $W$ on $[0,\tau_1]$ at the nodes $t_i= i/n\cdot \tau_1$, $i=0,\dots,n$, i.e.
 \[
  \overline{W}_n( t ) = \frac{ t - t_{i-1}}{\tau_1/n }
  \cdot
  W( t_i) + \frac{  t_i- t  }{\tau_1/n } \cdot W( t_{i-1}),\quad t\in[t_{i-1},t_i],
\]
for  $i\in\{1, \ldots, n\}$. We define approximations of the single components of $X(T)$ in the following way.  Put
\begin{equation}\label{87}
\begin{aligned}
\widehat X_{n, 1}(T)& =T, \quad \widehat X_{n, 2}(T)=-\int_0^{\tau_1}f'(t)\cdot \overline W_n(t)\,dt,\quad  \widehat X_{n, 3}(T)= \widehat X_{n, 2}^2(T),\\
\widehat X_{n, 4}(T) & = 0,\quad  \widehat X_{n, 5}(T)=\exp\bigl(\tfrac{1}{4p}\,
 \widehat X_{n, 2}^2\bigr),\quad \widehat X_{n, 6}(T)=0.
 \end{aligned}
\end{equation}
Next, let 
\[
a_\ell=2\sqrt{\ln \ell}
\]
for  $ \ell\in\N$ and put
\[
\widehat X^*_{n, 2}(T)=\sum_{\ell=1}^\infty \widehat X_{\ell n, 2}(T)\cdot 1_{[a_{\ell}, a_{\ell+1})}(|\widehat X_{n, 2}(T)|).
\]  
Finally, define  $G\colon \R\to\R$ by 
\begin{equation}\label{60}
G(x)=\frac{ \exp\bigl(\frac{1}{2p}
  x^2\bigr)}{(1+x^2)^{\frac{1}{2p}}\cdot \ln^{\frac{2}{p}} (2+x^2)},\quad x\in\R,
\end{equation}
and put
\[
\widehat X^\ast_{n, 7}(T)=G(\widehat X^*_{n, 2}(T))
\]
as well as
\[
 \widehat X^\ast_{n}(T)=\bigl(\widehat X_{n, 1}(T),\dots,\widehat X_{n, 6}(T), \widehat X^\ast_{n, 7}(T) \bigr).
\]

Clearly, the random number of evaluations of $W$  used by the approximation  
$\widehat X^\ast_{n}(T)$ is given by
\[
\text{cost}(\widehat X^\ast_n(T)) = n\sum_{\ell=1}^\infty \ell\cdot \1_{[a_{\ell}, a_{\ell+1})}(|\widehat X_{n, 2}(T)|).
\]

\begin{theorem}
\label{lt2}  There exists $c\in (0,\infty)$ such that for all $n\in\N$,
\[
\EE\bigl[\text{cost}(\widehat X^\ast_n(T))\bigr]\leq c\cdot n \quad\text{and}\quad
\bigl(\EE
    \bigl[|
      X( T ) -\widehat X^\ast_n(T)
    |^p\bigr]\bigr)^{\tfrac{1}{p}}\leq   \frac{c}{n}.
\]
\end{theorem}

Finally, we show that for $q<p$ a polynomial $q$-th mean error rate
for approximation of $X(T)$
 can be achieved with a sequence of non-adaptive approximations. For $n\in\N$ put 
\[
\widehat X_{n,7}(T) = G(\widehat X_{n,2}(T))
\]
with $G$ given by~\eqref{60} and define
\[
 \widehat X_{n}(T)=\bigl(\widehat X_{n, i}(T)\bigr)_{i=1,\dots,7}.
\]
Note that $\widehat X_{n}(T) = u_n(W(\tau_1/n),W(2\tau_1/n),\dots, W(\tau_1))$ for some function
$u_n\colon\R^n\to \R^7$.

\begin{theorem}
\label{Tt3}  Let $q\in[0,p)$. Then there exists $c\in (0,\infty)$ such that for all $n\in\N$,
\[
\bigl(\EE
    \bigl[|
      X( T ) -\widehat X_n(T)
    |^q\bigr]\bigr)^{\tfrac{1}{q}}\leq   \frac{c}{n}.
\]
\end{theorem}

\section{Proof of Theorem \ref{t1}}\label{proof1}
For the proof of Theorem~\ref{t1} we employ the following lemma, which is a straightforward generalization of Lemma 4.1 in  \cite{JMGY15}. 
\begin{lemma}\label{symm}
Let
$
  (\Omega_1, \mathcal{A}_1)
$
and
$
  (\Omega_2, \mathcal{A}_2)
$
be measurable spaces
and let
$
  V_1 \colon \Omega\to \Omega_1
$
and
$
  V_2, V_2', V_2''\colon$ $\Omega\to \Omega_2
$
be random variables such that
\begin{equation}
\label{eq:symm_ass}
  \PP_{ (V_1, V_2) } =
  \PP_{ (V_1, V_2') } =
  \PP_{ (V_1, V_2'') }
  \,
  .
\end{equation}
Then
for all $q\in[1, \infty)$  and for all measurable mappings
$
  \Phi \colon \Omega_1\times\Omega_2\to \R
$
and
$
  \varphi\colon \Omega_1\to\R,
$

\[
  \Bigl(\EE\big[
    |\Phi(V_1,V_2)- \varphi(V_1)|^q
  \big]\Bigr)^{\frac{1}{q}}
  \ge
  \frac{ 1 }{ 2 }
  \,
  \Bigl(\EE\big[
    | \Phi( V_1, V_2' ) - \Phi( V_1, V_2'' ) |^q
  \big]\Bigr)^{\frac{1}{q}}
  .
\]

\end{lemma}

\begin{comment}
\begin{proof}
Observe that
\eqref{eq:symm_ass}
ensures that
\[
  \EE\big[
    | \Phi(V_1, V_2) - \varphi(V_1) |^q
  \big]
  =
  \EE\big[
    |
      \Phi( V_1, V_2' ) - \varphi(V_1)
    |^q
  \big]
  =
  \EE\big[
    | \Phi(V_1,V_2'') - \varphi(V_1) |^q
  \big]
  .
\]
This and the  Minkowski's inequality imply that
\begin{align*}
  \Bigl(\EE\big[
    | \Phi(V_1,V_2) - \varphi(V_1) |^q
  \big]\Bigr)^{\frac{1}{q}}
&=\frac{ 1 }{ 2 } \,\Bigl(\Bigl(\EE\big[
    | \Phi(V_1,V_2') - \varphi(V_1) |^q
  \big]\Bigr)^{\frac{1}{q}}+\Bigl(\EE\big[
    | \Phi(V_1,V_2'') - \varphi(V_1) |^q
  \big]\Bigr)^{\frac{1}{q}}\Bigr) \\
&\geq
  \frac{ 1 }{ 2 } \,
  \Bigl(\EE\bigl[
    | \Phi(V_1,V_2') - \Phi(V_1,V_2'') |^q
  \bigr]\Bigr)^{\frac{1}{q}}
 ,
\end{align*}
  which finishes the proof of the lemma.
\end{proof}
\end{comment}

We start with the proof of Theorem \ref{t1}. 
Let $n\in \N$ and $s_1, \ldots, s_n\in [0,T]$. Clearly, there exist $0\le t_0 < t_1 \le T$ such that
\begin{equation}\label{Tpr1}
[t_0,t_1]\subset [0,\tau_1/2],\quad (t_0,t_1)\cap\{s_1,\dots,s_n\} = \emptyset,\quad t_1-t_0 = \tfrac{\tau_1}{2(n+1)}.
\end{equation}
Define  processes
$ \overline{W}, B \colon [ t_0, t_1 ] \times \Omega \to \R $
and
$ \widetilde{W} \colon \big( [ 0, t_0 ] \cup [ t_1, T ] \big) \times \Omega \to \R $
    by
\[
  \overline{W}( t ) = \frac{ (t - t_0) }{ ( t_1 - t_0 ) }
  \cdot
  W( t_1 ) + \frac{ ( t_1 - t ) }{ ( t_1 - t_0 ) } \cdot W( t_0 )
  ,
  \qquad
  B( t )
  = W( t ) - \overline{W}( t )
\]
   for $ t \in [ t_0, t_1 ] $   
and by
$
  \widetilde{W}( t ) = W( t )
$
for 
$ t \in [ 0, t_0 ] \cup [ t_1, T ] $. Moreover, let
\begin{align*}
Y_1&=-\int_{0}^{t_0} f'(s)\cdot W(s)\, ds-\int_{t_0}^{t_1} f'(s)\cdot \overline W(s)\, ds-\int_{t_1}^{\tau_1} f'(s)\cdot W(s)\, ds, \\
Y_2&=-\int_{t_0}^{t_1} f'(s)\cdot B(s)\, ds.
\end{align*}
By It\^{o}'s formula and~\eqref{assum1} we have $\PP$-a.s. 
\begin{equation}\label{Tpr1a}
Y_1+Y_2= \int_0^{\tau_1} f(s)\, dW(s).
\end{equation}
 Hence, by~\eqref{solution2}, $\PP$-a.s.
\begin{equation}\label{Tpr2}
X_7(T) = G(Y_1+Y_2),
\end{equation}
where $G\colon \R\to\R$ is given by \eqref{60}.

Let $u\colon\R^n\to\R$ be a  measurable mapping. Using~\eqref{Tpr2} we obtain 
\begin{equation}\label{Tpr3}
\EE
    \bigl[|
       X_7( T )-u(W(s_1), \ldots, W(s_n))
    |^p\bigr]=\EE
    \bigl[|
      G(Y_1+Y_2) -u(W(s_1), \ldots, W(s_n))
    |^p\bigr].
\end{equation}
The first two statements in~\eqref{Tpr1} imply that  
there exist measurable functions
$
  \Phi_1, \varphi \colon
  C\big(
    [ 0, t_0 ] \cup [ t_1 , T ] , \R
  \big) \to \R
$ and  $
  \Phi_2 \colon
  C\big(
    [t_0, t_1] , \R
  \big) \to \R
$
such that 
\[
  Y_1 =
  \Phi_1(
    \widetilde{W}
 ), \quad Y_2=\Phi_2(B), \quad u(W(s_1), \ldots, W(s_n))=\varphi(\widetilde{W}).
\] 
Moreover, $\widetilde W$ and $B$ are independent and $B$ has a symmetric distribution, which yields
\[
  \PP_{ ( \widetilde W, B ) }
%=  \PP_{ \widetilde W } \otimes \PP_{ B }
%=  \PP_{ \widetilde W } \otimes \PP_{ - B }
=
  \PP_{ ( \widetilde W, - B ) }.
\]
We may thus apply Lemma \ref{symm} with $
  \Omega_1 = C( [0, t_0] \cup [t_1 , T] , \R )
$,
$
  \Omega_2 = C( [t_0, t_1] , \R )
$, $V_1=\widetilde{W}$, $V_2=V_2'=B$, $V_2''=-B$, $\Phi(\tilde w,b)=G(\Phi_1(\tilde w)+\Phi_2(b))$ for $(\tilde w,b)\in\Omega_1\times \Omega_2$ and $\varphi$ as above, and observing the fact that $\Phi_2(-B)=-\Phi_2(B)$ we conclude that 
\begin{equation}\label{Tpr4}
\EE
    \bigl[|
      G(Y_1+Y_2) -u(W(s_1), \ldots, W(s_n))
    |^p\bigr]\geq \tfrac{1}{2^p}\, \EE
    \bigl[|
      G(Y_1+Y_2) -G(Y_1-Y_2)
    |^p\bigr].
\end{equation}

For the analysis of the right hand side in~\eqref{Tpr4} we first collect useful properties of the random variables $Y_1$ and $Y_2$ and the function $G$.  

Clearly, $Y_1$ and $Y_2$ are centered normal Gaussian variables. Moreover,  
independence of $\widetilde W$ and $B$ implies independence of $Y_1$ and $Y_2$.   
Let $\sigma_1^2$ and $\sigma_2^2$ denote the variances of $Y_1$ and $Y_2$, respectively. Due to \eqref{Tpr1a} and the fact that $\int_0^{\tau_1}f^2(t)dt=1$, see~\eqref{assum3}, we then have
\begin{equation}\label{Tpr5}
\sigma_1^2+\sigma_2^2=1.
\end{equation}
Put 
\[
\alpha=\inf_{t\in[0, \tau_1/2]} |f'(t)|^2, \quad \beta=\sup_{t\in[0, \tau_1/2]} |f'(t)|^2
\]
and note that $0< \alpha \le \beta <\infty$, due to~\eqref{assum2}. Since 
\[
\sigma_2^2 = \int_{t_0}^{t_1} \int_{t_0}^{t_1}f'(s)\cdot f'(t)\cdot \tfrac{(t_1-\max(s,t))(\min(s,t)-t_0)}{t_1-t_0}\, ds\,dt
\]
and $[t_0,t_1]\subset [0,\tau_1/2]$ we conclude that
\begin{equation}\label{Tpr6}
\alpha\cdot \tfrac{(t_1-t_0)^3}{12}\leq\sigma_2^2\leq \beta\cdot \tfrac{(t_1-t_0)^3}{12}.
\end{equation}
Put
\[
n_0=\big\lceil\tfrac{\tau_1}{2}\cdot \bigl(\tfrac{\beta}{6}\bigr)^{1/3}-1\big\rceil.
\]
Using~\eqref{Tpr1},~\eqref{Tpr5} and~\eqref{Tpr6} we obtain that if $n\ge n_0$ then 
\begin{equation}\label{Tpr7}
\tfrac{\alpha\cdot \tau_1^3}{96(n+1)^3} \le \sigma_2^2 \le 1/2 \le \sigma_1^2,\quad \sigma_1^{-2}-1 =\tfrac{\sigma_2^2}{\sigma_1^2}\le \tfrac{\beta\cdot  \tau_1^3}{48(n+1)^3}.
\end{equation}

Clearly, for all $x\ge 1$,
\begin{equation}\label{Tpr8}
G^p(x) =\tfrac{ \exp\bigl(
  x^2/2\bigr)}{(1+x^2)^{\frac{1}{2}}\cdot \ln^{2} (2+x^2)}\geq \tfrac{ \exp\bigl(
  x^2/2\bigr)}{\sqrt 2x\cdot \ln^{2} (3x^2)}.
\end{equation}
Moreover, $G$ is differentiable on $\R$ with
\begin{equation}\label{Tpr9}
G'(x)=\tfrac{x}{p}\cdot G(x)\cdot
  \bigl(1-\tfrac{1}{(1+x^2)}-\tfrac{ 4}{(2+x^2)\cdot \ln(2+x^2)}\bigr).
\end{equation}
Hence, for all $x\geq 3$,
\begin{equation}\label{Tpr10}
G'(x)\geq\tfrac{x}{2p}\cdot G(x)>0.
\end{equation}

Clearly, we may assume that $n\ge \max(3,n_0)$. Let $y_1\in [n^\frac{3}{2},2n^\frac{3}{2}]$ and $y_2\in [0,\sigma_2]$. Then
$y_1+y_2 \ge y_1-y_2 \ge n^\frac{3}{2}-\sigma_2\ge 3^\frac{3}{2} - 1 >3$, due to~\eqref{Tpr5}. Hence, by~\eqref{Tpr10},
\begin{align*}
|G(y_1+y_2)-G(y_1-y_2)| & \ge \int_{y_1}^{y_1+y_2} G'(x)\, dx \ge \tfrac{1}{2p}y_1\cdot y_2\cdot G(y_1),
\end{align*}
which jointly with~\eqref{Tpr8} yields
\[
|G(y_1+y_2)-G(y_1-y_2)|^p \ge \tfrac{1}{(2p)^p}\, y_1^{p-1}\cdot y_2^p\cdot \tfrac{\exp(y_1^2/2)}{\sqrt{2}\cdot \ln^2(3y_1^2)}.
\]
Employing~\eqref{Tpr5} and ~\eqref{Tpr7}  we conclude that
\begin{equation}\label{Tpr11}
\begin{aligned}
& \EE\bigl[|G(Y_1+Y_2)-G(Y_1-Y_2)|^p\bigr] \\
& \qquad\qquad \ge \frac{1}{(2p)^p 2^{\frac{3}{2}} \pi}\cdot\frac{1}{\sigma_1\sigma_2}\int_{n^{\frac{3}{2}}}^{2n^\frac{3}{2}}\int_0^{\sigma_2} \frac{y_1^{p-1}y_2^p}{\ln^2(3y_1^2)}\cdot\exp\Bigl(-\frac{y_2^2}{2\sigma_2^2}-\frac{y_1^2}{2}(\sigma_1^{-2}-1)\Bigr) \, dy_2\, dy_1\\
&  \qquad\qquad \ge \frac{1}{(2p)^p 2^{\frac{3}{2}} \pi}\cdot \frac{\sigma_2^{p}}{(p+1)\sqrt{e}}\int_{n^\frac{3}{2}}^{2n^\frac{3}{2}}\frac{y_1^{p-1}}{\ln^2(3y_1^2)}\cdot \exp\Bigl(-\frac{y_1^2}{2}(\sigma_1^{-2}-1)\Bigr) \,  dy_1\\
&  \qquad\qquad \ge \frac{1}{(2p)^p 2^{\frac{3}{2}} \pi}\cdot \frac{\sigma_2^{p}}{(p+1)\sqrt{e}}\cdot
\frac{n^{\frac{3p}{2}}}{\ln^2(12n^3)}\cdot \exp(-2n^3(\sigma_1^{-2}-1))\\
&  \qquad\qquad \ge \frac{1}{(2p)^p (p+1) 2^{\frac{3}{2}} \pi\sqrt{e}}\cdot \Bigl(\frac{\tau_1^3\alpha}{96}\Bigr)^{\frac{p}{2}}\cdot \Bigl(\frac{n}{n+1}\Bigr)^{\frac{3p}{2}}\cdot \frac{\exp\bigl(-\tfrac{n^3}{24(n+1)^3}\cdot\beta\tau_1^3\bigr)}{\ln^2(12n^3)}
\\
&  \qquad\qquad \ge \frac{(\tau_1^3\alpha)^{\frac{p}{2}}}{2^{5p+\frac{3}{2}}3^{\frac{p}{2}} p^p (p+1)  \pi\sqrt{e}}\cdot  \exp\bigl(-\tfrac{\beta\tau_1^3}{24}\bigr)\cdot \frac{1}{\ln^2(12n^3)}.
\end{aligned}
\end{equation}
Now combine~\eqref{Tpr3},~\eqref{Tpr4} and~\eqref{Tpr11} to complete the proof of Theorem~\ref{t1}.

\section{Proof of Theorems \ref{lt2} and~\ref{Tt3}} \label{proof2}

As  technical tools 
for the proof of Theorems \ref{lt2} and~\ref{Tt3} 
we employ the following two results for  centered Gaussian random variables.

\begin{lemma}\label{lemma2} 
For every $q\in [0,\infty)$ there exists $\kappa_q\in (0,\infty)$ such that for every 
random variable $Z\sim\mathcal{N}( 0,  \sigma^2)$ with $\sigma^2\in [0, \tfrac{1}{4}]$ and every 
$a\in [0,\infty)$,
\[
\EE \bigl[|Z|^q \cdot  \exp\bigl(a\cdot|Z|+|Z|^2\bigr)\bigr]\leq  \kappa_q\cdot \sigma^q\cdot (1+a^q)\cdot \exp(a^2\cdot\sigma^2).
\]
\end{lemma}

\begin{proof}
Let $q, a\in [0,\infty)$ and let $Z\sim\mathcal{N}( 0,  \sigma^2)$ with $\sigma^2\in [0, \tfrac{1}{4}]$. Without loss of generality we may assume that  $\sigma^2>0$. Let $V\sim \mathcal{N}( 0, 1)$.
 Then
\begin{align*}
 \EE \bigl[|Z|^q \cdot  \exp\bigl(a\cdot|Z|+|Z|^2\bigr)\bigr]
& =\tfrac{\sqrt 2}{\sqrt{\pi}\sigma}\cdot\int_{0}^{\infty} x^q\cdot \exp\bigl(a x+x^2-\tfrac{x^2}{2\sigma^2}\bigr)\,dx\\
& =\tfrac{\sqrt 2  }{\sqrt{\pi}\sigma}\cdot \exp\Bigl(\tfrac{a^2\sigma^2}{2-4\sigma^2}\Bigr)\cdot\int_{0}^\infty x^q\cdot \exp\bigl(-\tfrac{1-2\sigma^2}{2\sigma^2}\cdot\bigl(x-\tfrac{a\sigma^2}{1-2\sigma^2}\bigr)^2\bigr)\,dx\\
& \le \exp\bigl(\tfrac{a^2\sigma^2}{2-4\sigma^2}\bigr)\cdot \tfrac{2  }{\sqrt{1-2\sigma^2}}\cdot
\EE\bigl[\bigl|\tfrac{\sigma}{\sqrt{1-2\sigma^2}} \cdot V + \tfrac{a\sigma^2}{1-2\sigma^2}\bigr|^q \bigr]\\
& \le \exp\bigl(\tfrac{a^2\sigma^2}{2-4\sigma^2}\bigr)\cdot\tfrac{2^{q+1} \sigma^q }{(1-2\sigma^2)^{\frac{q+1}{2}}}\cdot \bigl(\EE[|V|^q] + a^q\tfrac{\sigma^q}{(1-2\sigma^2)^{\frac{q}{2}}}\bigr).
\end{align*}
Note that  $\sigma^2\le 1/4$ implies $1-2\sigma^2\ge 1/2$ as well as $\sigma^2/(1-2\sigma^2) \le 1/2$,
which finishes the proof of the lemma. 
\end{proof}

\begin{lemma}\label{lemma5}
Let $q\in[0,\infty)$ and $r\in [0,\tfrac{1}{2q})$ and let $H\in C^1(\R,\R)$ with
\begin{equation}\label{l90}
\sup_{x\in\R}|H'(x)|\cdot \exp(-q\cdot x^2) < \infty.
\end{equation}
Then there exists $\kappa\in (0,\infty)$ such that for all  independent random variables $V_1\sim\mathcal{N}( 0,  v_1^2)$, $V_2\sim\mathcal{N}( 0,  v_2^2)$ with $v_1^2+v_1^2 \le 1$, 
\[
\EE\bigl[|H(V_1+V_2)-H(V_1)|^r\bigr]\leq \kappa\cdot v_2^{r}.
\]
\end{lemma}

\begin{proof}
Let $q\in[0,\infty)$ and $r\in [0,\tfrac{1}{2q})$, let $H\in C^1(\R,\R)$ satisfy~\eqref{l90} and let $V_1\sim\mathcal{N}( 0,  v_1^2)$, $V_2\sim\mathcal{N}( 0,  v_2^2)$ be independent with $v_1^2+v_1^2 \le 1$. Let $U_1$ and $U_2$ be independent standard normal random variables. By the properties of $H$ there exists $c\in (0,\infty)$ such that for all $y,z\in\R$,
\[
|H(y+z)-H(y)|\leq \int_{\min(y,y+z)}^{\max(y,y+z)} |H'(x)|\,dx \leq c\cdot |z|\cdot\exp\bigl(q\cdot
  (|y|+|z|)^2\bigr).
\]
By the latter estimate, the H\"older inequality and the fact that $v_1^2 + v_2^2 \le 1$ we get 
\begin{equation}
\begin{aligned}
\EE\bigl[|H(V_1+V_2)-H(V_1)|^r\bigr] & \leq c^r\cdot v_2^r\cdot \EE\bigl[|U_2|^r\cdot\exp\bigl(r\cdot q\cdot
  (v_1\cdot|U_1|+v_2\cdot|U_2|)^2\bigr)\bigr]\\
  & \leq c^r\cdot v_2^r \cdot\EE\bigl[|U_2|^r\cdot\exp\bigl(r\cdot q\cdot
  (U_1^2+U_2^2)\bigr)\bigr] \\
  & = c^r\cdot v_2^r\cdot\EE\bigl[|U_2|^r\cdot\exp\bigl(r\cdot q\cdot
  U_2^2\bigr)\bigr]\cdot \EE\bigl[\exp\bigl(r\cdot q\cdot
  U_1^2\bigr)\bigr]\\
  & \le c^r\cdot v_2^r\cdot\Bigl(\EE\bigl[(1+|U_2|^r)\cdot\exp\bigl(r\cdot q\cdot
  U_2^2\bigr)\bigr]\Bigr)^2.
\end{aligned}
\end{equation}
Note that $r q< 1/2$ and put $v= (1-2rq)^{-1/2}$. Then
\[
\EE\bigl[(1+|U_2|^r)\cdot\exp\bigl(r\cdot q\cdot
  U_2^2\bigr)\bigr] = \int_{\R} \tfrac{(1+|x|^r)}{\sqrt{2\pi}} \cdot  \exp\bigl(\tfrac{-x^2}{2v^2}\bigr)\, dx = v\cdot\bigl(1+v^r\cdot \EE[|U_1|^r] \bigr),
\]
which completes the proof of the lemma.
\end{proof}

In the sequel we use the following notation. For $n\in\N$ we define $B_{n} \colon [ 0, \tau_1] \times \Omega \to \R $  by 
\[
B_{n}(t)=W(t)-\overline W_{n}(t),\quad t\in [0,\tau_1],
\]
and we put 
\begin{equation}\label{98}
Y_{n}=-\int_0^{\tau_1}f'(t) B_{n}(t)\,dt,\quad Z_n= \widehat X_{n, 2}(T)
\end{equation}
as well as
\[
\sigma_n^2 = \text{Var}(Y_n),\quad \nu_n^2=\text{Var}(Z_n).
\]
By It\^{o}'s formula and~\eqref{assum1} we have $\PP$-a.s. 
\begin{equation}\label{89}
X_2(T)= Z_n+Y_n.
\end{equation}

Let $n\in\N$ and $\ell\in\N$. Then it is easy to check that
\begin{equation}\label{T333} 
Z_{n},\, Z_{\ell n}-Z_{n},\, Y_{\ell n} \text{ are  independent, centered, Gaussian random variables}.
\end{equation}
Moreover, using~\eqref{89} and Lemma~\ref{Tlemma2} we get
\begin{equation}\label{Tvar}
\text{Var}(Y_{\ell n}) + \text{Var}(Z_{\ell n})  = \sigma^2_{\ell n} + \nu_{\ell n}^2 = 1,\quad
\text{Var}(Z_{\ell n}-Z_n)  = \nu^2_{\ell n} - \nu^2_{n},
\end{equation}
and, proceeding as in the proof of~\eqref{Tpr6}, it is easy to see that
\begin{equation}\label{Tvar2}
\sigma_{\ell n}^2 \le \gamma \tfrac{\tau_1^3}{12 \ell^2n^2}, 
\end{equation}
where $\gamma=\sup_{t\in[0, \tau_1]} |f'(t)|^2$.

We are ready to establish a $p$-th mean error estimate for the approximation $\widehat X^\ast_{n,7}(T)$.

\begin{lemma}\label{Tseven}
There exists $c\in (0,\infty)$  such that for all $n\in\N$,
\[
\EE    \bigl[|      X_7( T ) -\widehat X^\ast_{n,7}(T)    |^p\bigr]\leq   \frac{c}{n^p}.
\]
\end{lemma}

\begin{proof}

Let $n\in\N$. Using~\eqref{89} we obtain
\begin{equation}\label{71}
\begin{aligned}
\EE    \bigl[|      X_7( T ) -\widehat X^\ast_{n,7}(T)    |^p\bigr] &  =\EE
    \bigl[|
      G(X_2(T)) -G(\widehat X^*_{n, 2}(T))
    |^p\bigr] \\
    &   =\sum_{\ell=1}^\infty \EE
    \bigl[|
      G(Z_{\ell n}+Y_{\ell n}) -G(Z_{\ell n})
    |^p\cdot 1_{[a_{\ell}, a_{\ell+1})}(|Z_{n}|)\bigr].
    \end{aligned}
\end{equation}
It follows from~\eqref{Tpr9} that there exists $c_1\in (0,\infty)$ such that 
\begin{equation}\label{der1}
|G'(x)| \le c_1 \cdot|x|\cdot \exp(\tfrac{x^2}{2p})
\end{equation}
for all $x\in \R$. Hence, for all $y,z \in\R$, 
\[
|G(z+y)-G(z)|\leq \int_{\min(z,z+y)}^{\max(z,z+y)} |G'(x)|\,dx 
\leq c_1\cdot |y|\cdot (|z|+|y|)\cdot \exp\bigl(\tfrac{1}{2p}
  (|z|+|y|)^2\bigr),
\]
which jointly with \eqref{71} implies
\begin{align*}
& \EE    \bigl[|      X_7( T ) -\widehat X^\ast_{n,7}(T)    |^p\bigr]\\
& \qquad\quad \leq c_1^p\cdot \sum_{\ell=1}^\infty \EE
    \bigl[|Y_{\ell n}|^p\cdot (|Z_{\ell n}|+|Y_{\ell n}|)^p\cdot
      \exp\bigl(\tfrac{1}{2}
  (|Z_{\ell n}|+|Y_{\ell n}|)^2\bigr)\cdot 1_{[a_{\ell}, a_{\ell+1})}(|Z_{n}|)\bigr].
\end{align*}
Note that for all $\ell\in\N$,
\begin{align*}
     &  \exp\bigl(\tfrac{1}{2}
  (|Z_{\ell n}|+|Y_{\ell n}|)^2\bigr)\cdot 1_{[a_{\ell}, a_{\ell+1})}(|Z_{n}|)\\
  &\qquad \qquad \leq \exp\bigl(\tfrac{1}{2}
  |Z_{n}|^2+a_{\ell+1}\cdot(|Z_{\ell n}-Z_{n}|+|Y_{\ell n}|)+|Z_{\ell n}-Z_{n}|^2+|Y_{\ell n}|^2\bigr)\cdot 1_{{[a_{\ell}, a_{\ell+1})}}(|Z_{n}|)
\end{align*}
and
\[
(|Z_{\ell n}|+|Y_{\ell n}|)^p  
\leq 3^p\cdot (1+|Z_{n}|^p)\cdot (1+|Z_{\ell n}-Z_{n}|^p)\cdot (1+|Y_{\ell n}|^p).
\]
Hence, 
\begin{equation}\label{l24}
\EE    \bigl[|      X_7( T ) -\widehat X^\ast_{n,7}(T)    |^p\bigr]\leq (3c_1)^p\cdot \sum_{\ell=1}^\infty \EE
    \bigl[A_{\ell,n}\cdot B_{\ell,n}\cdot C_{\ell,n}\bigr],
\end{equation}
where 
\begin{align*}
A_{\ell,n} &=(1+|Z_{n}|^p)\cdot \exp\bigl(\tfrac{1}{2}
  |Z_{n}|^2\bigr)\cdot 1_{{[a_{\ell}, a_{\ell+1})}}(|Z_{n}|),\\
  B_{\ell,n}&=(1+|Z_{\ell n}-Z_{n}|^p)\cdot\exp\bigl(a_{\ell+1}\cdot |Z_{\ell n}-Z_{n}|+|Z_{\ell n}-Z_{n}|^2\bigr),\\
  C_{\ell,n}&=|Y_{\ell n}|^p\cdot (1+|Y_{\ell n}|^p)\cdot \exp\bigl(a_{\ell+1}\cdot|Y_{\ell n}|+|Y_{\ell n}|^2\bigr)
\end{align*}
for $\ell\in\N$. Observe that~\eqref{T333}  implies that for all $\ell\in\N$,
\begin{equation}\label{T444}
\EE
    \bigl[A_{\ell,n}\cdot B_{\ell,n}\cdot C_{\ell,n}\bigr] = \EE
    \bigl[A_{\ell,n}\bigr]\cdot \EE
    \bigl[ B_{\ell,n}\bigr]\cdot \EE
    \bigl[ C_{\ell,n}\bigr].
\end{equation}

Next, put 
\[
n_1 = \Bigl\lceil \sqrt{\gamma\tau_1^3}\Bigr\rceil.
\]
Using~\eqref{Tvar} and~\eqref{Tvar2} we see that  for all $n\ge n_1$ and $\ell\in\N$,
\begin{equation}\label{T334}
\nu^2_{\ell n} - \nu^2_{n} = \sigma_n^2- \sigma_{\ell n}^2 \le \sigma_n^2 \le \tfrac{\gamma\tau_1^3}{12 n^2} \le \tfrac{1}{12}.
\end{equation}
Using~\eqref{T334}, ~\eqref{Tvar2} 
and Lemma~\ref{lemma2} we thus obtain that there exist $\kappa_0,\kappa_p,\kappa_{2p},c_2,c_3\in (0,\infty)$ such that for all $n\ge n_1$ and $\ell\in\N$,
\[
\EE
    \bigl[ B_{\ell,n}\bigr]  \le \bigl(2\kappa_0 + \kappa_p\cdot (\nu_{\ell n}^2- \nu_n^2)^{\frac{p}{2}}\cdot (1+a_{\ell+1}^p)\bigr)\cdot \exp\bigl(a_{\ell+1}^2\cdot (\nu_{\ell n}^2- \nu_n^2)\bigr) \le c_3\cdot  \ln^{\frac{p}{2}}(\ell+1)
\]
and
\[
\EE
    \bigl[ C_{\ell,n}\bigr]  \le \bigl(\kappa_p\cdot \sigma_{\ell n }^p\cdot  (1+a_{\ell+1}^p) + \kappa_{2p}\cdot \sigma_{\ell n }^{2p}\cdot (1+a_{\ell+1}^{2p})\bigr)\cdot \exp\bigl(a_{\ell+1}^2\cdot \sigma_{\ell n }^{2}\bigr)\le  \tfrac{c_2}{(\ell n)^p} \cdot \ln^{p}(\ell+1)\cdot \ell^{\tfrac{1}{3}}.
\]
Furthermore,~\eqref{Tvar} and~\eqref{T334} jointly imply that $11/12\le\nu_n^2 \le 1$ for all $n\ge n_1$, and therefore there exists $c_4\in (0,\infty)$ such that for all $n\ge n_1$ and $\ell\in\N$,
\begin{align*}
\EE
    \bigl[ A_{\ell,n}\bigr] & = \tfrac{2}{\sqrt{2\pi}}\int_{a_{\ell}/\nu_n}^{a_{\ell+1}/\nu_n} (1+(\nu_n\cdot x)^p)\cdot \exp\bigl(-\tfrac{x^2}{2}\cdot (1-\nu_n^2)\bigr)\,dx\leq  (1+a_{\ell+1}^p)\cdot \tfrac{ a_{\ell+1}-a_{\ell}}{\nu_n}\\& \le \sqrt{\tfrac{12}{11}} 
\bigl( 1 + 2^p\ln^{\frac{p}{2}}(\ell+1)
   \bigr)\cdot \tfrac{2}{\ell\cdot \ln^{\frac{1}{2}}(\ell+1) }\le c_4 \cdot \ln^{\frac{p-1}{2}}(\ell+1) \cdot \tfrac{1}{\ell}.
\end{align*}
Hence, there exists $c_5\in (0,\infty)$ such that for all $n\ge n_1$ and $\ell\in\N$,
\begin{equation}\label{T569}
\begin{aligned}
\EE
    \bigl[A_{\ell,n}\bigr]\cdot \EE
    \bigl[ B_{\ell,n}\bigr]\cdot \EE
    \bigl[ C_{\ell,n}\bigr] \le c_5 \cdot \tfrac{1}{n^p} \cdot\tfrac{\ln^{2p-\frac{1}{2}}(\ell+1)}{\ell^{p+\frac{2}{3}}}. 
\end{aligned}
\end{equation} 

Combining~\eqref{l24},~\eqref{T444} and~\eqref{T569} we conclude that there exists $c_6\in (0,\infty)$ such that for all $n\ge n_1$,
\begin{equation}\label{T2345}
\EE    \bigl[|      X_7( T ) -\widehat X^\ast_{n,7}(T)    |^p\bigr] \le (3c_1)^p \cdot c_5 \cdot \tfrac{1}{n^p} \sum_{\ell=1}^\infty \tfrac{\ln^{2p-\frac{1}{2}}(\ell+1)}{\ell^{p+\frac{2}{3}}} \le c_6 \cdot\tfrac{1}{n^p}.
\end{equation}

In view of Lemma~\ref{Tlemma2} it remains                 
to prove that for all $n<n_1$,
\begin{equation}\label{T566}
\EE
    \bigl[|G(\widehat X^*_{n, 2}(T))
    |^p\bigr] < \infty.
\end{equation}

To this end we define $\rho\colon\R\to [3,\infty)$ by $\rho(x) = \max(|x|,3)$. Clearly, $\rho$ is convex. Moreover, by~\eqref{Tpr9} and~\eqref{Tpr10} we obtain that $G$ as well as $G'$ are increasing on $[3,\infty)$. In particular, $G$ is convex on $[3,\infty)$. Using the monotonicity and convexity of $G$ as well as the convexity of $\rho$ we conclude that $G\circ \rho$ is convex.  

For $n,\ell\in\N$  put $\mathcal F_{\ell n} = \sigma(\{W(i\tau_1/(\ell n))\colon i = 1,\dots, \ell n\})$ and note that $Z_{\ell n} = \EE[X_2(T)|  \mathcal F_{\ell n} ]$. Using the estimate $G(x)\le (\ln(2))^{-2/p}\exp(x^2/2p)$ and the Jentzen inequality we therefore obtain that for all $n,\ell\in\N$
\begin{equation}\label{T111}
\begin{aligned}
& |G(Z_{\ell n})|^p \cdot \1_{[a_{\ell}, a_{\ell+1})}(|Z_{n}|)\\
& \qquad \le  \bigl(|G(\rho(Z_{\ell n}))|^p+(\ln(2))^{-2}\exp(9/2)\bigr) \cdot \1_{[a_{\ell}, a_{\ell+1})}(|Z_{n}|)\\
& \qquad \le \bigl(\EE[|G(\rho(X_2(T)))|^p|  \mathcal F_{\ell n} ] +(\ln(2))^{-2}\exp(9/2)\bigr)\cdot \1_{[a_{\ell}, a_{\ell+1})}(|Z_{n}|)\\
& \qquad = \EE\bigl [|G(\rho(X_2(T)))|^p \1_{[a_{\ell}, a_{\ell+1})}(|Z_{n}|)|  \mathcal F_{\ell n} \bigr] +(\ln(2))^{-2}\exp(9/2)\cdot \1_{[a_{\ell}, a_{\ell+1})}(|Z_{n}|).
\end{aligned}
\end{equation}
Hence there exists $c_7\in (0,\infty)$ such that for all $n\in\N$,
\[
|G(\widehat X^*_{n, 2}(T))|^p \le c_7\cdot \Bigl(\sum_{\ell=1}^\infty  \EE\bigl [|G(\rho(X_2(T)))|^p \1_{[a_{\ell}, a_{\ell+1})}(|Z_{n}|)|  \mathcal F_{\ell n} \bigr] +1\Bigr),
\]
which implies
\[
\EE\bigl[|G(\widehat X^*_{n, 2}(T))|^p\bigr]  \le c_7\cdot \bigl(\EE\bigl[|G(\rho(X_2(T)))|^p\bigr] + 1\bigr).
\]
Finally, note that $G(\rho(X_2(T)))\le G(X_2(T)) + (\ln(2))^{-2/p}\exp(9/2p)$ and apply Lemma~\ref{Tlemma2} to complete the proof of~\eqref{T566}.
\end{proof}

Next, we provide error estimates for the approximations $\widehat X_{n,2}(T),\widehat X_{n,3}(T),\widehat X_{n,5}(T)$ and $\widehat X_{n,7}(T)$. 

\begin{lemma}\label{Teight} Let $r_2,r_3\in (0,\infty)$, $r_5\in (0,2p)$ and $r_7\in (0,p)$. Then there exists $c\in (0,\infty)$ such that for every $i\in\{2,3,5,7\}$ and every $n\in\N$,
\[
\EE\bigl[|X_i(T)-\widehat X_{n,i}(T)|^{r_i}\bigr]\leq \frac{c}{n^{r_i}}.
\]
\end{lemma}

\begin{proof}
Let $r_2,r_3\in (0,\infty)$, $r_5\in (0,2p)$ and $r_7\in (0,p)$. Let $n\in\N$. 
In all of the four cases we apply Lemma~\ref{lemma5} with $V_1 = \widehat X_{n,2}(T) = Z_n$ and $V_2=Y_n$, see~\eqref{89}. Thus $\text{Var}(V_1) +  \text{Var}(V_2) = 1$, and according to~\eqref{Tvar2} we have $ \text{Var}(V_2) \le (\gamma\tau_1^3)/(12n^2)$.
Moreover, for $i\in \{2,3,5,7\}$ we use the function $H=H_i$ in Lemma~\ref{lemma5}, where $H_2,H_3,H_5,H_7\colon \R\to \R$ are given by
\[
H_2(x) = x,\quad H_3(x) = x^2,\quad H_5(x) =\exp\bigl(\tfrac{1}{4p}x^2\bigr),\quad H_7(x) =
G(x).
\]
Let 
\[
q_2=0, \quad q_3\in (0,1/(2r_3)),\quad q_5\in (1/(4p),1/(2r_5)),\quad q_7\in (1/(2p),1/(2r_7)).
\]
Employing~\eqref{der1} in the case $i=7$ we then see that there exists $c\in (0,\infty)$ such that for every $i\in\{2,3,5,7\}$ and every $x\in\R$,
\[
|H_i'(x)| \le c\cdot \exp(q_i\cdot x^2),
\]
which completes the proof.
\end{proof}

Clearly, Lemmas~\ref{Tseven} and \ref{Teight} jointly yield the error estimates in Theorems~\ref{lt2} and~\ref{Tt3}.

It remains to establish the cost estimate in Theorem~\ref{lt2}. 

Let $n\in\N$.
Clearly, if $\nu_n=0$ then $\widehat X_{n,2}(T)=Z_n=0$ a.s. and we have $\text{cost}(\widehat X^*_n(T))  = n$ a.s. Next, assume $\nu_n^2 >0$. Using~\eqref{T333} and the fact that $\nu_n^2\le 1$ we get for every $l\in\N$, 
\begin{equation}\label{T767}
\begin{aligned}
\PP(\{|Z_n|\in  [a_\ell, a_{\ell+1})\}) & =\tfrac{2}{\sqrt{2\pi}\nu_n}\int_{a_{\ell}}^{a_{\ell+1}}  \exp\bigl(-\tfrac{x^2}{2\nu_n^2}\bigr)dx\leq \tfrac{2}{\sqrt{2\pi}\nu_n}\cdot  \exp\bigl(-\tfrac{a_l^2}{2\nu_n^2}\bigr)\cdot (a_{\ell+1}-a_{\ell})\\
& \le \tfrac{1}{\nu_n}\cdot \tfrac{\sqrt{2}}{\sqrt{\pi}} \cdot \exp(-2\ln(\ell))\cdot \tfrac{2}{\sqrt{\ln(2)}\cdot \ell}\le \tfrac{1}{\nu_n}\cdot\tfrac{2^{3/2}}{\sqrt{\pi \ln(2)}} \cdot \tfrac{1}{\ell^3}.
\end{aligned}
\end{equation}
Hence 
\[
\EE\bigl[\text{cost}(\widehat X^*_n(T))\bigr] =  
\sum_{\ell=1}^\infty \ell\cdot n\cdot \PP(\{|Z_n|\in [a_\ell, a_{\ell+1})\})\leq \tfrac{n}{\nu_n} \cdot \tfrac{\sqrt{2}}{\sqrt{\pi \ln(2)}} \cdot \sum_{\ell=1}^\infty \tfrac{1}{\ell^2}.
\]
Finally note that $\lim_{n\to\infty}\nu_n^2 = 1$, due to~\eqref{Tvar} and~\eqref{Tvar2}, and therefore $\displaystyle{\inf_{n\colon \nu_n >0} \nu_n > 0}$, which completes the proof of the cost estimate and finishes the proof of Theorems~\ref{lt2} and~\ref{Tt3}.

\section{Discussion}
\label{discuss}

The key contribution of this paper is to show that even then when an autonomous SDE on $[0,T]$ has smooth coefficients with first order derivatives of at most linear growth and its solution $X$ satisfies $\EE[\sup_{t\in [0,T]} |X(t)|^p]< \infty$, where $p\in [1,\infty)$, it may happen that $X(T)$ can not be approximated on the basis of finitely many observations of the driving Brownian motion at fixed times in $[0,T]$ with a polynomial $p$-th mean error rate, see Theorem~\ref{t1}.
This result naturally leads to a number of questions related to possible extensions or tightenings with respect to the class of approximations, the speed of convergence, the moment conditions on the solution and the polynomial growth conditions on the first order derivatives of the coefficients. 

Does there exist an SDE of the above type such that a sub-polynomial rate of convergence holds for any adaptive approximation as well? For the SDE considered in the present paper, there is an adaptive method, which achieves a polynomial error rate, see Theorem~\ref{lt2}.

Does there exist an SDE of the above type such that the smallest possible $p$-th mean error  that can be achieved by any non-adaptive 
method based on $n$ evaluations of the driving Brownian motion
or even 
by any
adaptive method based on $n$ evaluations of the driving Brownian motion on average converges to zero slower than a given arbitrarily slow decay in terms of $n$? A negative result of this type is true for the class of  SDEs with bounded smooth coefficients, see \eqref{eq:intro3} and \cite{JMGY15,GJS17,Y17}.

The first order derivatives of the coefficients of the pathological SDE \eqref{hardsde3} constructed in the present paper are of at most linear growth. 
Can a sub-polynomial rate of convergence of the smallest possible $p$-th mean error also
occur when the first order derivatives of the coefficients are of at most polynomial growth with an exponent $\alpha\in (0,1)$?

Finally it is open, whether a sub-polynomial rate of convergence of the smallest possible $p$-th mean error can also occur when the solution $X$ has finite moments of some order $q>p$ or even satisfies $\EE[\sup_{t\in[0, T]}|X(t)|^q] < \infty$ for all $q\ge 1$. The pathological SDE \eqref{hardsde3} satisfies $\EE[\sup_{t\in[0, T]}|X(t)|^q] < \infty$ only for $q\le p$,
    see Lemma \ref{Tlemma2}.

\bibliographystyle{acm}
\bibliography{bibfile}

\end{document}